\documentclass[a4paper,11pt]{amsart}

\usepackage{amsmath,amssymb,amsthm,hyperref}
\usepackage[margin=2.5cm]{geometry}

\newtheorem*{thm}{Theorem}
\newtheorem*{cor}{Corollary}
\newtheorem*{ack}{Acknowledgement}

\title{Abelian subgroups of Helly groups}
\author{Motiejus Valiunas}
\address{Instytut Matematyczny, Uniwersytet Wroc{\l}awski, plac Grunwaldzki 2/4, 50-384 Wrocław, Poland}
\email{Motiejus.Valiunas@math.uni.wroc.pl}

\begin{document}

\begin{abstract}
We observe that abelian subgroups of Helly groups are finitely generated, and consequently, soluble subgroups of Helly groups are virtually abelian.
\end{abstract}
\maketitle

A simplicial graph $\Gamma$ is called \emph{Helly} if for any collection $\{ B_i \mid i \in I \}$ of balls in $\Gamma$, if $B_i \cap B_j \neq \varnothing$ for all $i,j \in I$ then $\bigcap_{i \in I} B_i \neq \varnothing$. We say a group $G$ is \emph{Helly} if it acts geometrically on a Helly graph.

The class of Helly groups has been introduced and studied in \cite{ccgho}. It is a large class of groups, including all hyperbolic, cocompactly cubulated, and $C(4)$--$T(4)$ small cancellation groups \cite[Theorem~1.1]{ccgho}. On the other hand, all Helly groups are biautomatic and have some other non-positive-curvature-like behaviour \cite[Theorem~1.5]{ccgho}.

Motivated by this, we may ask whether all abelian subgroups of Helly groups are finitely generated \cite[Question~9.3]{ccgho}, which is a special case of the long-standing open question on whether all abelian subgroups of biautomatic groups are finitely generated \cite[\S7]{gersten-short}. In this note, we observe that the results of Descombes \& Lang \cite{descombes-lang} can be combined with those of Haettel \& Osajda \cite{haettel-osajda} to give an affirmative answer to the former question. We expect this result to be well-known to experts, but it does not seem to appear anywhere in the literature.

\begin{thm}
Abelian subgroups of Helly groups are finitely generated.
% BONUS PART FOR ANYONE WHO'S DOWNLOADED THE .tex VERSION
% In particular, if $G$ is a group acting geometrically on a Helly graph $\Gamma$ of combinatorial dimension $N$, then there exists a finite group $F$ such that every abelian subgroup of $G$ embeds in $\mathbb{Z}^N \times F$.
\end{thm}

Before we prove the Theorem, let us note one immediate consequence.

\begin{cor}
Soluble subgroups of Helly groups are virtually abelian.
\end{cor}

\begin{proof}
Let $G$ be a Helly group. Since, by the Theorem, abelian subgroups of $G$ are finitely generated, it follows that soluble subgroups of $G$ are polycyclic \cite[Theorem~2 on p.~25]{segal}. On the other hand, as $G$ is biautomatic \cite[Theorem~1.5(1)]{ccgho}, polycyclic subgroups of $G$ are virtually abelian \cite[Theorem~6.15]{gersten-short}.
\end{proof}

\begin{proof}[Proof of the Theorem]
Let $G$ be a Helly group, so that $G$ acts geometrically on a Helly graph $\Gamma$. By \cite[Theorem~6.3]{ccgho}, $G$ also acts geometrically on the \emph{injective hull} $E\Gamma$ of $\Gamma$---a locally finite polyhedral complex of dimension $N < \infty$.

Let $H \leq G$ be an abelian subgroup, and let $n = \dim_{\mathbb{Q}} (H \otimes \mathbb{Q}) \in \mathbb{Z}_{\geq 0} \cup \{\infty\}$ be the torsion-free rank of $H$. Then $H$ (and so $G$) has a subgroup isomorphic to $\mathbb{Z}^k$ for any non-negative integer $k \leq n$. On the other hand, $E\Gamma$ admits a $G$-equivariant consistent bicombing \cite[Proposition~3.8]{lang}, implying by the Flat Torus Theorem \cite[Theorem~1.2]{descombes-lang} that $E\Gamma$ contains an isometrically embedded $k$-dimensional normed real vector space, and in particular that $k \leq N$, whenever $G$ contains a free abelian subgroup of rank $k$. Therefore, we have $n \leq N$, and in particular $n < \infty$.

Now consider the translation length function on $H$ with respect to the action on $\Gamma$, defined by $\|h\| := \lim_{m \to \infty} \frac{d_\Gamma(x,h^m \cdot x)}{m}$ for any $h \in H$, where $x \in \Gamma$ is a choice of a basepoint. It is clear that $\|h^m\| = |m| \cdot \|h\|$ and (since $H$ is abelian) $\|gh\| \leq \|g\|+\|h\|$ for any $g,h \in H$ and $m \in \mathbb{Z}$, implying that $\|{-}\|$ induces a semi-norm on $H \otimes \mathbb{Q} \cong \mathbb{Q}^n$. On the other hand, for any infinite order element $h \in H$, the translation length $\|h\|$ is non-zero and rational with denominator bounded above by $2N$ \cite[Theorem~O]{haettel-osajda}, implying that $\|h\| \geq \frac{1}{2N}$. Therefore, the image of the map $p\colon H \to H \otimes \mathbb{Q}$ is a discrete subgroup of $H \otimes \mathbb{Q} \cong \mathbb{Q}^n$, and therefore isomorphic to a lattice in $\mathbb{R}^n$, implying that $p(H) \cong \mathbb{Z}^n$.

Finally, $\ker(p)$ is a torsion subgroup of the Helly group $G$, implying by \cite[Corollary~I]{haettel-osajda} that $\ker(p)$ is finite. This shows that $H$ is finitely generated, as required.
% PROOF OF THE BONUS PART
% Therefore, by the structure theorem for finitely generated abelian groups we have $H \cong \mathbb{Z}^n \times \ker(p)$. The final assertion follows since $G$ has finitely many conjugacy classes of finite subgroups \cite[Theorem~1.5(2)]{ccgho}, so we may take $F$ to be a finite group containing all finite subgroups of $G$ (up to isomorphism).
\end{proof}

\begin{ack}
The author would like to thank Damian Osajda for a valuable discussion.
\end{ack}

\bibliographystyle{amsalpha}
\bibliography{../../all}

\end{document}